\documentclass[review,hidelinks,onefignum,onetabnum]{siamart220329}
\usepackage{bbm}
\usepackage{geometry} 
\nolinenumbers 

\newcommand{\xdownarrow}[1]{%
	{\left\downarrow\vbox to #1{}\right.\kern-\nulldelimiterspace}
}
\usepackage{lipsum}
\usepackage{amsfonts}
\usepackage{graphicx}
\usepackage{epstopdf}
\usepackage{algorithmic}
\ifpdf
  \DeclareGraphicsExtensions{.eps,.pdf,.png,.jpg}
\else
  \DeclareGraphicsExtensions{.eps}
\fi

\newsiamremark{remark}{Remark}
\newsiamremark{hypothesis}{Hypothesis}
\crefname{hypothesis}{Hypothesis}{Hypotheses}
\newsiamthm{claim}{Claim}

\headers{Coupled stochastic systems of Skorokhod type}{T. K. T. Thieu, A. Muntean, and R. Melnik}

\title{Coupled stochastic systems of Skorokhod type: well-posedness of a mathematical model and its applications\thanks{Submitted to the editors DATE.}}

\author{Thi Kim Thoa Thieu \thanks{MS2Discovery Interdisciplinary Research Institute, Wilfrid Laurier University, 75 University Ave W, Waterloo, Ontario, Canada N2L 3C5.} 
	\and  Adrian Muntean\thanks{Department of Mathematics and Computer Science and Center for Societal Risk Research (CSR), Karlstad University, Karlstad, Sweden.} 
		\and Roderick Melnik \footnotemark[2] \thanks{BCAM - Basque Center for Applied Mathematics, Bilbao, Spain.}}

\usepackage{amsopn}

\ifpdf
\hypersetup{
  pdftitle={Coupled stochastic systems of Skorokhod type: well-posedness of a mathematical model and its applications},
  pdfauthor={T. K. T. Thieu, A. Muntean, and R. Melnik}
}
\fi

\begin{document}

\maketitle

\begin{abstract}
	Population dynamics with complex biological interactions,
accounting for uncertainty quantification, is critical for many application areas. However, due to the complexity of biological systems, the mathematical formulation of the corresponding problems faces the challenge that the corresponding stochastic processes should, in
most cases, be considered in bounded domains. 
We propose a model based on a coupled system of reflecting Skorokhod-type stochastic differential equations with  jump-like exit from a boundary. The setting describes the population dynamics of active and passive populations. As main working techniques, we use compactness methods and Skorokhod's representation of solutions to SDEs posed in bounded domains to prove the well-posedness of the system. This functional setting is a new point of view in the field of modelling and simulation of population dynamics. We provide the details of the model, as well as representative numerical examples, and discuss the applications of a Wilson-Cowan-type system, modelling the dynamics of two interacting populations of excitatory and inhibitory neurons. Furthermore, the presence of random input current, reflecting factors together with Poisson jumps, increases firing activity in neuronal systems. 
\end{abstract}



	\section{Introduction}\label{intro}

	In recent years, the modelling of population dynamics arising from biological systems offers many challenging questions to some of the most advanced areas of science and technology. In order to model reliably population dynamics, accounting for the complexity in the interactions among populations, considering for uncertainty quantification is critical for many application areas. One of the main challenges is that the corresponding stochastic processes should, in
most cases, be considered in bounded domains. Before studying such stochastic models in detail, the question of well-posedness has to be addressed. There are several results for stochastic differential equations (SDEs) with reflecting boundary conditions \cite{Tanaka1979,Slominski1994,MarinRubio2004,Wells2006,Sabelfeld2019}, one of them being the seminal contribution of Skorokhod in \cite{Skorohod1961}, where the author provided the existence and uniqueness to one-dimensional stochastic equations for diffusion processes in a bounded region. A direct approach to the solution of the reflecting boundary conditions and reductions to the case of nonsmooth domains are reported in \cite{Lions1984}. Extending results by Tanaka, the author of \cite{Saisho1987}  proved the existence and uniqueness of solutions to the Skorokhod equation posed in a bounded domain in $\mathbb{R}^d$ where a reflecting boundary condition was applied. In \cite{Dupuis1993}, the authors studied the strong existence and uniqueness of the stochastic differential equations with reflecting boundary conditions for domains that might have conners. In addition, the existence, uniqueness and stability of solutions of multidimensional SDE's with reflecting boundary conditions were provided in \cite{Slominski1993}, where the author obtained results on the existence and uniqueness of strong and weak solutions to the SDE for any driving semimartingale and in a more general domain. 
The models of stochastic differential equations in a bounded domain have been 
known for a long time and yet, only a few relevant results are available in the context of population dynamics for the problems posed in confined domains. As far as we are aware, one of the first questions in this setting was posed in the modelling and simulation study \cite{Richardson2019}  while considering the evacuation dynamics of a mixed active-passive pedestrian populations in a complex geometry in the presence of a fire as well as of a slowly spreading smoke curtain. From a stochastic processes perspective, various lattice gas models for active-passive pedestrian dynamics have been recently explored in \cite{Cirillo2019} and \cite{Colangeli2019}. See also \cite{Thieu2019} for a result on the weak solvability of a deterministic system of parabolic partial differential equations describing the interplay of a mixture of flows for active-passive populations of pedestrians. In general, the purely diffusive Brownian motion with random fluctuations of continuous sample paths is used to be assumed as noise in a dynamical system. However, the diffusive fluctuations are large and abrupt events that appear at random times throughout the time series. Therefore, the description of such diffusive fluctuations is incomplete to demonstrate real population dynamics, and the jump-diffusion stochastic processes provide a more accurate descriptions for population dynamics models \cite{Slominski2010,Bressloff2011,Melanson2019}. 

Motivated by \cite{Pilipenko2012,Przybylowicz2021,Situ2005}, we are interested in a  coupled system of reflecting Skorokhod-type stochastic differential equations with jumps, modelling the dynamics of active and passive populations. In this paper, we prove the well-posedness of a coupled system of reflecting Skorokhod-type stochastic differential equations with jump-like exit from a boundary for active-passive population dynamics. From the modelling perspective, our approach is novel, opening new routes for investigation of population dynamics, including the computability of solutions and identification of model parameters. Taking the inspirations from the applications of population dynamics and neuroscience \cite{Wilson1972,Kilpatrick2014}, we provide details of an application of our active-passive population dynamics model in a Wilson-Cowan-type system describing the dynamics of two interacting populations of excitatory and inhibitory neurons.

	\section{Mathematical model: coupled stochastic processes in bounded domain}\label{Model}
We start from considering the dynamics of active-passive population dynamics. Each population is considered in a  one-dimensional domain, then the whole system will be embeded in a two-dimensional domain, which we refer to as $\Omega$. Let $D\subset \mathbb{R}$ satisfies the assumption $(\text{A}_2)$ in Section \ref{Skorohod}. We denote $S=(0,T)$ for some $T \in \mathbb{R}_{+}$. We refer to $\bar{D}\times S$ as $D_T$, note that $\bar{D}$ denotes the closure of $D$. 
\subsection{Active particle population}\label{active}
Our main focus in the remainder of this section is to find an explicit formula for a solution of the reflection problem, which is similar to the Skorokhod-like problem but involves the possibility of a jump-like exit from zero.
For $t \in S$, $f_A: D_T \longrightarrow \mathbb{R},$ and $g_A: D_T \times D_T \longrightarrow \mathbb{R}^{2}$, let $X_{A}$ denote the active population at time $t$. We assume that the dynamics of active population is governed by the following model (see, e.g., \cite{Pilipenko2012,Przybylowicz2021})

\begin{align}\label{Active_population}
	\begin{cases}
		dX_{A}(t) &= f_{A}(X_{A}(t), X_{P}(t))dt + g_{A}(X_{A}(t), X_{P}(t))dW_{A}(t) +d\phi_{A}(t) \\&+ \int_{\mathbb{R}} \rho_A(X_A(s),y)\nu(dy,ds)  \\
		X_{A}(0) &= X_{A,0}, 
	\end{cases}
\end{align}	
where $\rho_A$ is a measurable function such that $\rho_A: \mathbb{R} \to \mathbb{R}$, $W_A = (W_A(t))_{t\in [0,T]}$ is a two-dimensional standard Brownian motion, while $\nu$ is a Poisson random measure with finite jump intensity, associated with a scalar compound Poisson process (clarified below by relationship \eqref{compoundPossion}).
\subsection{Passive particle population}
The case of passive particle populations is treated in a way similar described above.
For $t \in S$, $f_P: D_T \longrightarrow \mathbb{R},$ and $g_P: D_T \times D_T \longrightarrow \mathbb{R}^{2}$, let $X_{P}$ denote the passive population inside the domain $\bar{D}$. The dynamics of the passive population is described  here by a system of stochastic differential equations as follows (see, e.g., \cite{Pilipenko2012,Przybylowicz2021}):
\begin{align}\label{Passive_population}
	\begin{cases}
		dX_{P}(t) &= f_{P}(X_{P}(t), X_{P}(t))dt + g_{P}(X_{A}(t), X_{P}(t))dW_{P}(t) +d\phi_{P}(t)  \\&+ \int_{\mathbb{R}} \rho_P(X_P(s),y)\nu(dy,ds) \\
		X_{P}(0) &= X_{P,0}, 
	\end{cases}
\end{align}
where $\rho_P$ is a measurable function $\rho_P: \mathbb{R} \to \mathbb{R}$, $W_P = (W_P(t))_{t\in [0,T]}$ is a 2-dimensional standard Brownian motion, while $\nu$ is a Poisson random measure with finite jump intensity, associated with a scalar compound Poisson process (clarified below by relationship \eqref{compoundPossion}).

The proposed dynamics \eqref{Active_population}-\eqref{Passive_population} are the general structures of our active-passive population dynamics. We will discuss further detailed model descriptions as well as the applications of the population dynamics in Section \ref{applications}.

	\subsection{The Skorokhod equation}\label{Skorohod}

Now, having representations for active and passive populations, we would like to consider a system of stochastic Skorokhod-type equations and analyze their properties. We consider the following equation (see, e.g., \cite{Pilipenko2012,Przybylowicz2021})
\begin{align}
	dX(t) &= df(X_A(t),X_P(t))dt + dg(X_A(t),X_P(t))dW(t) + d\Phi(t) \nonumber\\&+ \int_{\mathbb{R}} \rho(X(s),y)\nu(dy,ds)
\end{align}

Let $\{W(t), t\geq 0\}$ be a Wiener process and let $\{J(t), t\geq 0\}$ be a nondecreasing L\'{e}vy process independent of $\{W(t), t \geq 0\}$ with finite L\'{e}vy measure $\mu$. The jump measure $\nu$ is a Poisson random measure with finite jump intensity, associated with a compound Poisson process that can be represented by the following form
\begin{align}\label{compoundPossion1}
	J(t) = \sum_{n=1}^{N(t)} \xi_k,
\end{align}
where $N(t)$, $t\geq 0$ is a Poisson process with intensity $\alpha = \mu((0,\infty))$ and $\{\xi_n, n \geq 1\}$ are independent identically distributed random variables independent of $N(t)$, $t\geq 0$ such that $P(\xi_n \leq x) = \mu((0,x])/\alpha$.



\subsubsection{Assumptions}\label{assumption}

We rely on the following assumptions: 

\begin{itemize}
	\item[($\text{A}_1$)] The functions $f: D_T\times D_T \longrightarrow \mathbb{R}\times \mathbb{R},$ and $g: D_T \times D_T \longrightarrow \mathbb{R}^{2} \times \mathbb{R}^{2}$ satisfy the global Lipschitz conditions. 
	\item[($\text{A}_2$)] $\partial D$ is $C^{2,\alpha}$ with $\alpha \in (0,1)$. 
	\item[($\text{A}_3$)] There exists a constant $C_\rho \in (0,\infty)$ such that the jump coefficient $\rho: \mathbb{R}\times \mathbb{R} \to \mathbb{R}$ satisfies the following inequality for all $x \in \mathbb{R}$ (see, e.g., \cite{Przybylowicz2021}):
	
	\begin{align}
		\int_{\mathbb{R}}\|\rho(x,y)\|^2\psi(dy) \leq C_\rho(1 + \|x\|^2)
	\end{align}
	and for all $x,z \in \mathbb{R}$,
	\begin{align}
		\int_{\mathbb{R}}\|\rho(x,y) - \rho(z,y)\|^2\psi(dy) \leq C_\rho\|x - z\|^2,
	\end{align}
	where $\psi$ is the distribution of $\xi$. Moreover, $\rho$ is a bounded function.
\end{itemize}

	It is worth mentioning that assumptions ($\text{A}_1$) and ($\text{A}_2$) correspond to the modeling of the situation in Section \ref{applications}, while ($\text{A}_3$)-($\text{A}_6$) are of technical nature, corresponding to the type of solution we are searching for; clarifications in this direction are given in the next Sections.

Let us denote $\rho(x,y) = y \cdot \rho(x)$. Then, the jump process can be considered as a compound Poisson process, that is for all $t \in [0,T]$ we have  (see, e.g., \cite{Przybylowicz2021}):
\begin{align}\label{compoundPossion}
	\int_{0}^{t} \int_{\mathbb{R}} \rho(X(s),y)\nu(dy,ds) &= \sum_{k=1}^{N_t}\xi_k\cdot \rho(X(\tau_k)) \nonumber\\&= \sum_{0<s\leq t}\rho(X(s)) \Delta J(s)  = \int_0^t \rho(X(s))dJ(s).
\end{align}
\subsubsection{Concept of solution}

Take $x \in \partial D$ arbitrarily fixed. We define the set $\mathcal{N}_x$ of inward normal unit vectors at $x \in \partial D$ by
\begin{align}
	\mathcal{N}_{x}&= \cup_{r > 0}\mathcal{N}_{x,r}, \nonumber\\
	\mathcal{N}_{x,r}&=\left\{\textbf{n} \in \mathbb{R}^2: |\textbf{n}|=1, B(x-r\textbf{n},r) \cap D = \emptyset \right\},
\end{align}
where $B(z,r) = \{y \in \mathbb{R}^2: |y-z| < r\}, z \in \mathbb{R}^2, r>0$. Mind that, in general, it can happen that $\mathcal{N}_x = \emptyset$. In this case, the uniform exterior sphere condition is not satisfied (see, for instance,  the examples provided in Fig. 5 in \cite{Cholaquidis2016} and in page $4$ in \cite{Choulli2016}). 
\newline
We complement our list of  assumptions  ($\text{A}_1$)--($\text{A}_3$) with three specific conditions on the geometry of the domain $D$:

\begin{itemize}
	\item[($\text{A}_4$)] (Uniform exterior sphere condition). There exists a constant $r_0 > 0$ such that
	$$\mathcal{N}_x = \mathcal{N}_{x,r_0} \neq \emptyset \text{ for any } x \in \partial D.$$
	\item[($\text{A}_5$)] There exist constants $\delta > 0$ and $\delta' \in [1, \infty)$ with the following property: for any $x \in \partial D$ there exists a unit vector $\textbf{l}_x$ such that 
	$$\langle\textbf{l}_x,\textbf{n}\rangle \geq 1/\delta' \text{ for any } \textbf{n} \in \bigcup_{y \in B(x, \delta) \cap \partial D} \mathcal{N}_y,$$ 
	where $\langle\cdot, \cdot\rangle$ denotes the usual inner product in $\mathbb{R}^2$.
	\item[($\text{A}_6$)]  There exist $\delta''>0$ and $\nu>0$ such that for each $x_0 \in \partial D$ we can find a function $f \in C^2(\mathbb{R}^2)$ satisfying 
	\begin{align}
		\langle y-x, \textbf{n}\rangle + \frac{1}{\nu}\langle \nabla f(x), \textbf{n}\rangle|y-x|^2 \geq 0,
	\end{align}
	for any $x \in B(x_0, \delta'') \cap \partial D, y \in B(x_0, \delta'') \cap \partial \bar{D}$ and $\textbf{n} \in \mathcal{N}_x$.
\end{itemize}

The following relation is called {\em the Skorokhod equation}: Find $(\xi, \phi) \in C(\bar{S}, \mathbb{R}^2)\times C(\bar{S})$ such that
\begin{align}\label{deterministicS}
	\xi(t)= w(t) + \phi(t),
\end{align}
where $w \in C(\bar{S}, \mathbb{R}^2)$ is given so that $w(0) \in \bar{D}$. 
The solution of \eqref{deterministicS} is a pair $(\xi,\phi)$, which satisfies the following two conditions:
\begin{itemize}
	\item[(a)] $\xi \in C(\bar{S}, \bar{D})$; 
	\item[(b)] $\phi \in C(\bar{S})$ with bounded variation on each finite time interval satisfying $\phi(0)=0$ and 
	\begin{align}
		\phi(t) &= \int_0^t\textbf{n}(y)d|\phi|_y,\nonumber\\
		|\phi|_{t} &= \int_0^t\mathbbm{1}_{\partial D}(\xi(y)) d|\phi|_y,
	\end{align}
	where 
	\begin{align}\label{BV-func}
		\textbf{n}(y) &\in \mathcal{N}_{\xi(y)} \text{ if } \xi(y) \in \partial D, \nonumber\\
		|\phi|_t &= \text{ total variation of } \phi \text{ on } [0,t] \nonumber\\
		&=\sup_{\mathcal{T} \in \mathcal{G}([0,t])} \sum_{k=1}^{n_{\mathcal{T}}} |\phi(t_k) - \phi(t_k-1)|.
	\end{align}
	In \eqref{BV-func}, we denote by $\mathcal{G}([0,t])$ the family of all partitions of the interval $[0,t]$ and take a partition $\mathcal{T} = \{0=t_0<t_1< \ldots <t_{n_\mathcal{T}} =t\} \in \mathcal{G}([0,t])$. The supremum in \eqref{BV-func} is taken over all partitions of type $0=t_0<t_1< \ldots <t_{n_\mathcal{T}} =t$. 
\end{itemize}
Conditions (a) and (b) guarantee that $\xi$ is a \textit{reflecting process} on $\bar{D}$.

It is easily seen from the definition that
$$\xi_1(t) = w_1(t), \ldots, w_{d-1}(t) = \xi_{d-1},$$
and 
$$\xi_d(t)) = w_d(t)+\phi(t), \int_0^t\mathbbm{1}_{\xi_d(y) \notin \partial D}d|\phi|_y.$$
We define a multidimensional Skorokhod's map $\Gamma: C(\bar{S}) \longrightarrow C(\bar{S})$ such that
\begin{align}
	\Gamma w(t) = \Gamma(w_1, \ldots, w_d)(t) = (\Gamma w_1(t), \ldots, \Gamma w_{d-1}(t), \Gamma w_d(t)).
\end{align}
Hence, the pair $(\xi_d, \phi)$ is the exact solution of the one-dimensional Skorokhod problem $\xi_d$. Therefore, it holds
\begin{align}\label{reflection-func}
	\phi(t) = -\min_{y \in [0,t]}\{w_d(y),0\}, \quad \ \xi_d(t) =  w_d(t) - \min_{y \in [0,t]}\{w_d(y),0\} = \Gamma w_d(t).
\end{align}
The multidimensional Skorokhod's map $\Gamma$ satisfies the Lipschitz condition in a space of continuous functions.
\begin{theorem}\label{deterministic-theo}
	Assume conditions ($\text{A}_4$) and ($\text{A}_5$). Then for any $w \in C(\bar{S}, \mathbb{R}^2)$ with $w(0) \in \bar{D}$, there exists a unique solution $\xi(t,w)$ of the equation \eqref{deterministicS} such that $\xi(t,w)$ is continuous in $(t,w)$.
\end{theorem}
For the proof of this Theorem, we refer the reader to Theorem $4.1$ in \cite{Saisho1987}.

To come closer to the model equations for active-passive population dynamics described above in Sections \ref{active}-\ref{Skorohod},  we introduce the mappings
$$f: D_T \times D_T \longrightarrow \mathbb{R}^2 , \quad g: D_T \times D_T \longrightarrow \mathbb{R}^{2\times 2} $$
and consider the Skorokhod-like system on the probability space $(\Omega, \mathcal{F}, P)$
\begin{align}\label{main-skorohod}
	dX(t) = f(X(t))dt +g(X(t))dW(t) + \int_{\mathbb{R}} \rho(X(s),y)\nu(dy,ds) + d\Phi(t)
\end{align}
with
\begin{align}\label{initial-cond}
	X(0)=X_0 \in \bar{D},
\end{align}
where the inital value $X_0$ is assumed to be an $\mathcal{F}_0-$measurable random variable and $B(t)$ is a $2\times2-$dimensional $\mathcal{F}_t-$Brownian motion with $B(0)=0$. Here, $\{\mathcal{F}_t\}$ is a filtration such that $\mathcal{F}_0$ contains all $P-$negligible sets and $\mathcal{F}_t= \cap_{\varepsilon>0}\mathcal{F}_{t+\varepsilon}$, while $\Phi_t(t)$ is defined in the Definition \ref{def} below. Further properties of the structure of \eqref{main-skorohod}-\eqref{initial-cond}  are listed in Section \ref{wellpossedness}. Similarly to the deterministic case, we can now define the following concept of solutions to \eqref{main-skorohod}.
\begin{definition}\label{def}
	A pair $(X(t),\Phi(t))$ is called solution to \eqref{main-skorohod}--\eqref{initial-cond} if the following conditions hold:
	
	\begin{itemize}
		\item[(i)] $X(t)$ is a $\bar{D}-$valued $\mathcal{F}_t-$adapted continuous process;
		\item[(ii)] $\Phi(t)$ is an $\mathbb{R}^2-$valued $\mathcal{F}_t-$adapted continuous process with bounded variation on each finite time interval such that $\Phi(0)=0$ with
		\begin{align}
			\Phi(t) &= \int_0^t\textbf{n}(y)d|\Phi|_y,\nonumber\\
			|\Phi|(t) &= \int_0^t\mathbbm{1}_{\partial D}(X(y)) d|\Phi|_y.
		\end{align} 
		\item[(iii)] $\textbf{n}(s) \in \mathcal{N}_{X(s)} \in \partial D$.
	\end{itemize}
\end{definition}
Note that the Definition \ref{def} ensures that $X(t)$ entering \eqref{main-skorohod}-\eqref{initial-cond} is a reflecting process on $\bar{D}$.

	\section{Well-posedness of Skorokhod-type system}\label{wellpossedness}

In this section, we establish the well-posedness of the Skorokhod-type system by showing the existence, uniqueness and stability of solutions to the problem \eqref{main-skorohod}--\eqref{initial-cond} in the sense of Definition \ref{def}.

We use the compactness method together with the continuity result of the deterministic case stated in Theorem \ref{deterministic-theo} for proving the existence of solutions to \eqref{main-skorohod}-\eqref{initial-cond}. We follow the arguments  by G. Da Prato and J. Zabczyk  ($2014$) (cf. \cite{Prato14}, Section $8.3$) and a result of F. Flandoli (1995) (cf. \cite{Flandoli95}) for martingale solutions. The starting point of this argument is based on considering a sequence $\{X_t^n\}$ of solutions of the following system of Skorokhod-type stochastic differential equations
\begin{align}\label{main-skorohod2}
	\begin{cases}
		dX^n(t) &= f(X^n(h^n(t)))dt + g(X^n(h^n(t)))dW(t) + \int_{\mathbb{R}} \rho(X^n(s),y)\nu(dy,ds) \nonumber\\&+ d\Phi^n(t),\\
		X^n(0) &= X_0 \in \bar{D},
	\end{cases}
\end{align}
where $X_0^n \in \bar{D}$ is given.

For convenience, we recast the solution to the system \eqref{Active_population} and \eqref{Passive_population} in terms of the vector $X_t^n(t)$, $n\in \mathbb{N}$, such that
\begin{align*}
	X^n(t) &:= (X_A^n(t), X_P^n(t)),
	f(X^n(t)) := (f_A(X_A^n(t), X_P^n(t)),f_P(X_A^n(t), X_P^n(t))),\\
	g(X^n(t)) &:= (g_A(X_A^n(t), X_P^n(t)),g_P(X_A^n(t), X_P^n(t))),
	\Phi^n(t) := (\phi_{A}^n(t),\phi_{P}^n(t)),\\
	J(t) &:= (J_A(t), J_P(t)),
	X^n(0) := (X_{A,0},X_{P,0}),
	X(t) := (X_A(t), X_P(t)),\\
	f(X(t)) &:= (f_A(X_A(t), X_P(t)),f_P(X_A(t), X_P(t))),\\
	g(X(t)) &:= (g_A(X_A(t), X_P(t)),g_P(X_A(t), X_P(t))),\\
	\Phi(t) &:= (\phi_{A}(t),\phi_{P}(t)),
	J(t) := (J_A(t), J_P(t)), W(t):= (W_A(t), W_P(t)).
\end{align*}

Let us introduce the following step functions:
\begin{align}\label{step-func1}
	h^n(0)=0,
\end{align}
\begin{align}
	h^n(t)= (k-1)2^{-n},
\end{align}
\begin{align}\label{step-func2}
	(k-1)2^{-n} <t\leq  k 2^{-n}, \quad k=1,2,\ldots,n \text{ and } n\geq 1.
\end{align}
Using Theorem \ref{deterministic-theo}, we have a unique solution of \eqref{main-skorohod2}. Furthermore, each value of $X_t^n(t)$ is obtained within $0 \leq t \leq k2^{-n}$ and $X_t^n(t)$ is attained for $k2^{-n} < t \leq (k+1)2^{-n}$ that is uniquely determined as the solution of the following Skorokhod equation: 
\begin{align}
	X^n(t) &= X^n(k2^{-n}) + f(X^n(k2^{-n}))(t-k2^{-n}) + g(X^n(k2^{-n}))\{W(t) - W(k2^{-n})\} \nonumber\\&+ \int_{\mathbb{R}} \rho(X^n(k2^{-n}),y)\nu(dy,ds) + \Phi^n(t).
\end{align}
Let us denote
\begin{align}
	Y^n(t) &:= X_0 + \int_{0}^tf(X^n(h^n(y)))dy + \int_0^t g(X^n(h^n(y)))dW(y) \nonumber\\&+ \int_{0}^{t}\int_{\mathbb{R}} \rho(X^n(k2^{-n}),y)\nu(ds,dy). 
\end{align}
Then $X^n(t) = (\Gamma Y_t^n)(t)$, we also have
\begin{align}\label{subsystem}
	Y^n(t) &:= X_0 + \int_{0}^tf((\Gamma Y^n)(h^n(y)))dy \nonumber\\&+ \int_0^t g((\Gamma Y^n)(h^n(y)))dW(y) + \int_{0}^{t}\int_{\mathbb{R}} \rho(X^n(h^n(s)),y)\nu(ds,dy). 
\end{align}
We define the family of laws
\begin{eqnarray}\label{laws}
	\left\{\mathcal{P}(Y^n); 0\leq t \leq T, n\geq 1\right\}.
\end{eqnarray}
Accordingly, \eqref{laws} is a family of probability distributions of $Y^n$. Let $\mathcal{P}^n$ be the laws of $Y^n$.
\subsection{Statement of the main theoretical results of the paper}

The main theoretical results of this paper are stated in Theorems \ref{main-theo1}-\ref{main-stabi} below. In this section, the focus lies on ensuring the well-posedness of Skorokhod solutions with jump-like exit from a boundary to our population dynamics problem.

\begin{theorem}[Existence]\label{main-theo1}
	Assume that $(\text{A}_1)$-$(\text{A}_3)$ hold. There exists at least a weak solution to the Skorokhod-type system \eqref{main-skorohod}--\eqref{initial-cond} in the sense of Definition \ref{def}.
\end{theorem}
\begin{theorem}[Uniqueness]\label{main-theo2} Assume that $(\text{A}_1)$-$(\text{A}_6)$ hold.
	There is a unique strong solution to \eqref{main-skorohod}--\eqref{initial-cond}.
\end{theorem}
\begin{theorem}[Dependence on parameters]\label{main-stabi}
	Assume that $(\text{A}_1)$-$(\text{A}_3)$ hold and 
	\begin{align}
		\lim_{k \to \infty} E(|X_0^k - X_0|^2) = 0.
	\end{align}
	Suppose that $X_t^n(t) \in C(\bar{S}; \bar{D}\times \bar{D})$ solves
	\begin{align}
		\begin{cases}
			dX^n(t) = f(X^n(t))dt + g(X^n(t))dW(t) + \int_{\mathbb{R}} \rho(X^n(s),y)\nu(ds,dy) +  d\Phi^n(t) ,\\
			X^n(0) = X_0^n \in \bar{D},
		\end{cases}
	\end{align}
	where $X_0^n \in \bar{D}$ is given and $\Phi^n(t)$ is defined as a sequence of $\Phi(t)$ in Definition \ref{def}.
	Then
	\begin{align}
		\lim_{n \to \infty} E(\max_{0 \leq t \leq T}|X^n(t) - X(t)|^2) = 0,
	\end{align}
	where $X_t \in C(\bar{S}; \bar{D}\times \bar{D})$ is the unique solution of the following problem:
	\begin{align}
		\begin{cases}
			dX(t) = f(X(t))dt + g(X(t))dW(t) +\int_{\mathbb{R}} \rho(X(s),y)\nu(ds,dy) + d\Phi(t),\\
			X(0) = X_0 \in \bar{D}.
		\end{cases}
	\end{align}
\end{theorem}
These statements are proven in the next subsections \ref{existence-proof}-\ref{stability-proof}.

\subsubsection{Proof of the existence}\label{existence-proof}
Let us start with handling the tightness of the laws $\{\mathcal{P}^n\}$ through the following Lemma.
\begin{lemma}\label{tightness}
	Assume that $(\text{A}_1)$-$(\text{A}_3)$ hold. Then, the family $\{\mathcal{P}^n\}$ given by \eqref{laws} is tight in $C(\bar{S}, \mathbb{R}^2\times \mathbb{R}^2 )$.
\end{lemma}
\begin{proof}
	Let us introduce the following relative compact set in $C(\bar{S}, \mathbb{R}^2 \times \mathbb{R}^2)$
	\begin{align}
		K_{M_1M_2} = \left\{f \in C(\bar{S}; \mathbb{R}^2 \times \mathbb{R}^2): \|f\|_{L^\infty(S; \mathbb{R}^2 \times \mathbb{R}^2)} \leq M_1, [f]_{C^\alpha(\bar{S};\mathbb{R}^2 \times \mathbb{R}^2)} \leq M_2\right\}.
	\end{align}
	Now, we will show that for a given $\varepsilon > 0$, there are $M_1, M_2 > 0$ such that
	\begin{align}
		P(Y^n(\cdot) \in K_{M_1M_2}) \leq \varepsilon, \text{ for all } n \in \mathbb{N}.
	\end{align}
	This means that
	\begin{align}
		P(\|Y^n(t)\|_{L^\infty(S;\mathbb{R}^2 \times \mathbb{R}^2)} > M_1 \text{ or } [Y^n(\cdot)]_{C^\alpha(\bar{S};\mathbb{R}^2 \times \mathbb{R}^2)} > M_2) \leq \varepsilon.
	\end{align}
	A sufficient condition for this to happen is
	\begin{align}
		P(\|Y^n(t)\|_{L^\infty(S; \mathbb{R}^2\times \mathbb{R}^2)} > M_1) < \frac{\varepsilon}{2} \text{ and } P([Y^n(\cdot)]_{C^\alpha(\bar{S}; \mathbb{R}^2\times \mathbb{R}^2)} > M_2) < \frac{\varepsilon}{2},
	\end{align}
	where $Y(\cdot)$ denotes either $Y(t)$ or $Y(r)$.
	
	We consider first $	P(\|Y^n(\cdot)\|_{L^\infty(S, \mathbb{R}^2\times \mathbb{R}^2)} > M_1) < \frac{\varepsilon}{2} $. Using Markov's inequality (see e.g. \cite{Jacod2004}, Corollary 5.1), we get
	\begin{align}
		P(\|Y^n(t)\|_{L^\infty(S; \mathbb{R}^2\times \mathbb{R}^2)} > M_1) \leq \frac{1}{M_1}E[\sup_{t\in \bar{S}} |Y^n(t)|],
	\end{align}
	but 
	{\small\begin{align}
		\sup_{t \in \bar{S}}|Y^n(t)| &= \sup_{t \in \bar{S}} \Bigg\{\left|X_{A,0} + \int_{0}^t f_A((\Gamma Y^n)(h^n(y)))dy+ \int_0^tg_A((\Gamma Y^n)(h^n(y)))dW_A(y)\right|,\nonumber\\& \left|X_{P,0} + \int_{0}^t f_P((\Gamma Y^n)(h^n(y)))dy + \int_0^tg_P((\Gamma Y^n)(h^n(y)))dW_P(y)\right|\Bigg\}.
	\end{align}}
	We estimate 
{\small	\begin{align}\label{sup1}
		\sup_{t \in \bar{S}}|Y^n(t)| &\leq \sup_{t \in \bar{S}} \Bigg\{\left|X_{A,0}\right| + \left|\int_{0}^t f_A((\Gamma Y^n)(h^n(y)))dy\right|+ \left|\int_0^tg_A((\Gamma Y^n)(h^n(y)))dW(y)\right|, \nonumber\\& \left|X_{P,0}\right| + \left|\int_{0}^t f_P((\Gamma Y^n)(h^n(y)))dy\right| + \left|\int_0^tg_P((\Gamma Y^n)(h^n(y)))dW(y)\right|\Bigg\}.
	\end{align}}
	Since $F_1, F_2$ are bounded, then we have 
	\begin{align}
		\int_{0}^T f_A((\Gamma Y^n)(h^n(y)))dy \leq C \text{ and } \int_{0}^T f_P((\Gamma Y^n)(h^n(y)))dy \leq C.
	\end{align}
	Taking the expectation on \eqref{sup1}, we are led to 
	\begin{align}
		E\left[ \sup_{t\in \bar{S}} |Y^n(t)|\right] \leq C + E\left[\sup_{t \in \bar{S}}\left|\int_0^t g((\Gamma Y^n)(h^n(y)))dW(y)\right|\right].
	\end{align}
	On the other hand, the Burlkholder-Davis-Gundy's inequality \footnote[1]{See e.g. \cite{Karatzars2000}, Theorem 3.28 (The Burlkholder-Davis-Gundy's inequality). Let $M \in \mathcal{M }^{c,\text{loc}}$ and $M_t^{*} := \max_{0\leq s\leq t}|M_s|$. For every $m > 0$, there exists universal positive constants $k_m$, $K_m$ (depending only on $m$), such that the inequalities $$k_m E(\langle M\rangle_T^m \leq E[(M_T^{*})^{2m}] \leq K_mE(\langle M\rangle_T^m)$$ hold for every stopping time $T$. Note that $\mathcal{M }^{c,\text{loc}}$ denotes the space of continuous local martingales and $\langle X\rangle$ represents the quadratic variance process of $X \in \mathcal{M }^{c,\text{loc}}$.}  implies
	\begin{align}
		E\left[\sup_{t \in \bar{S}}\left|\int_0^t g((\Gamma Y^n)(h^n(y)))dB(y)\right|\right]  &\leq K_{1/2}E\left[\int_0^t |g((\Gamma Y^n)(h^n(y)))|^2dy\right]^{1/2} \nonumber\\
		&\leq K_{1/2}t^{1/2},
	\end{align}
	for $K_{1/2} >0$.\\
	Then, we have the following estimate
	\begin{align}
		E\left[ \sup_{t\in \bar{S}} |Y^n(t)|\right]  \leq C+ E\left[\int_0^t |g((\Gamma Y^n)(h^n(y)))|^2dy\right]^{1/2} \leq C.
	\end{align}
	Hence, for $\varepsilon > 0$, we can choose $M_1>0$ such that $	P(\|Y^n\|_{L^\infty(S; \mathbb{R}^2\times \mathbb{R}^2)} > M_1) < \frac{\varepsilon}{2} $.
	
	In the sequel, we consider the second inequality $P([Y^n(\cdot)]_{C^\alpha(\bar{S}; \mathbb{R}^2\times \mathbb{R}^2)} > M_2) < \frac{\varepsilon}{2}$, this reads
	\begin{align}
		P([Y^n(\cdot)]_{C^\alpha(\bar{S};\mathbb{R}^2 \times \mathbb{R}^2)} > M_2) = P\left(\sup_{t\neq r; t,r \in \bar{S}} \frac{|Y^n(t) - Y^n(r)|}{|t-r|^{\alpha}} > M_2\right) \leq \frac{\varepsilon}{2}.
	\end{align}
	
	Let us introduce another class of compact sets now in the Sobolev space \\$W^{\alpha, p}(0,T; \mathbb{R}^2\times \mathbb{R}^2)$ (which for suitable exponents $\alpha p - \gamma > 1$ lies in $C^\gamma(\bar{S}; \mathbb{R}^2\times \mathbb{R}^2)$). Additionally, we recall the relatively compact sets $K'_{M_1M_2}$, defined as in \ref{appendix}, such that
	\begin{align}\label{km1m2}
		K'_{M_1M_2} = \left\{ f \in C(\bar{S};\mathbb{R}^2 \times \mathbb{R}^2): \|f\|_{L^{\infty}(S; \mathbb{R}^2 \times \mathbb{R}^2)} \leq M_1, [f]_{W^{\alpha,p}(S; \mathbb{R}^2 \times \mathbb{R}^2)} \leq M_2\right\},
	\end{align}
	where $\alpha p>1$ (see e.g. \cite{Flandoli95}, \cite{Colangeli2019}). Having this in mind, we wish to prove that there exits $\alpha \in (0,1)$ and $p>1$ with $\alpha p>1$ together with the property: given $\varepsilon > 0$, there is $M_2>0$ such that 
	\begin{align}
		P([Y^n(\cdot)]_{W^{\alpha, p}(S; \mathbb{R}^2 \times \mathbb{R}^2)} > M_2) < \frac{\varepsilon}{2} \text{ for every } n \in \mathbb{N}.
	\end{align}
	Using Markov's inequality, we obtain 
	\begin{align}
		P([Y^n(\cdot)]_{W^{\alpha, p}(S; \mathbb{R}^2 \times \mathbb{R}^2)} > M_2) &\leq \frac{1}{M_2} E\left[\int_0^T \int_0^T \frac{|Y^n(t) - Y^n(r)|^p}{|t-r|^{1+\alpha p}} dt dr\right] \nonumber\\
		&= \frac{C}{M_2} \int_0^T \int_0^T \frac{E[|Y^n(t) - Y^n(r)|^p]}{|t-r|^{1+\alpha p}} dt dr.
	\end{align}
	For $t>r$, we have
	\begin{align}
		Y^n(t) - Y^n(r) &= \begin{bmatrix}
			\int_r^tf_A((\Gamma Y_A^n)(h^n(y)),(\Gamma Y_P^n)(h^n(y)))dy\\\int_r^t f_P((\Gamma Y_A^n)(h^n(y)),(\Gamma Y_P^n)(h^n(y)))dy
		\end{bmatrix} \\&+
		\begin{bmatrix}
			\int_r^t g_P((\Gamma Y_A^n)(h^n(y)),(\Gamma Y_P^n)(h^n(y))) dW_A(y)\\\int_r^t g_P((\Gamma Y_A^n)(h^n(y)),(\Gamma Y_P^n)(h^n(y))) dW_P(y)
		\end{bmatrix}\\&
		+\begin{bmatrix}
			\int_r^t \rho_A(y) dJ_A(y)\\\int_r^t \rho_P(y) dJ_P(y)
		\end{bmatrix}.
	\end{align}
	Let us introduce some further notation. For a vector $u=(u_1,u_2)$, we set $|u|:= |u_1|+|u_2|$. At this moment, we consider the following expression 
	\begin{align}
		|Y^n(t) - Y_r^n(r)| &= \Bigg| \int_{r}^t f_A((\Gamma Y_A^n)(h^n(y)),(\Gamma Y_P^n)(h^n(y)))dy \nonumber\\&+  \int_{r}^t g_A((\Gamma Y_A^n)(h^n(y)),(\Gamma Y_P^n)(h^n(y))) dW_A(y)\Bigg|\nonumber\\&+ \Bigg| \int_{r}^t f_P((\Gamma Y_A^n)(h^n(y)),(\Gamma Y_P^n)(h^n(y)))dy\nonumber\\&+\int_r^tg_P((\Gamma Y_A^n)(h^n(y)),(\Gamma Y_P^n)(h^n(y))) dW_P(y)\Bigg| \nonumber\\
		&+ \left| \int_{r}^t\rho_A((\Gamma Y_A^n)(h^n(y))) dJ_A(y) +\int_r^t\rho_P((\Gamma Y_P^n)(h^n(y))) dJ_P(y)\right|.
	\end{align}
	Taking the modulus up to the power $p>1$ and applying Minkowski inequality, we have 
\begin{align}\label{modulus2}
		|Y_t^n(t) - Y_r^n(r)|^p &= \Bigg(\Bigg| \int_{r}^t f_A((\Gamma Y_A^n)(h^n(y)),(\Gamma Y_P^n)(h^n(y)))dy \nonumber \\&+  \int_{r}^t g_A((\Gamma Y_A^n)(h^n(y)),(\Gamma Y_P^n)(h^n(y))) dW_A(y)\Bigg| \nonumber \\&+ \Bigg| \int_{r}^t f_P((\Gamma Y_A^n)(h^n(y)),(\Gamma Y_P^n)(h^n(y)))dy\nonumber \\&+\int_r^tg_P((\Gamma Y_A^n)(h^n(y)),(\Gamma Y_P^n)(h^n(y))) dW_P(y)\Bigg| \nonumber\\&+ \left| \int_{r}^t\rho_A((\Gamma Y_A^n)(h^n(y))) dJ_A(y) +\int_r^t\rho_P((\Gamma Y_P^n)(h^n(y))) dJ_P(y)\right| \Bigg)^p\nonumber \\
		&\leq C\Bigg(\Bigg| \int_{r}^t f_A((\Gamma Y_A^n)(h^n(y)),(\Gamma Y_P^n)(h^n(y)))dy \nonumber \\&+  \int_{r}^t g_A((\Gamma Y_A^n)(h^n(y)),(\Gamma Y_P^n)(h^n(y))) dW_A(y)\Bigg|^p \nonumber \\&+ \Bigg| \int_{r}^t f_P((\Gamma Y_A^n)(h^n(y)),(\Gamma Y_P^n)(h^n(y)))dy\nonumber \\&+\int_r^tg_P((\Gamma Y_A^n)(h^n(y)),(\Gamma Y_P^n)(h^n(y))) dW_P(y)\Bigg|^p\nonumber\end{align}
	\begin{align}&+ \left| \int_{r}^t\rho_A((\Gamma Y_A^n)(h^n(y))) dJ_A(y) +\int_r^t\rho_P((\Gamma Y_P^n)(h^n(y))) dJ_P(y)\right|^p\Bigg) \nonumber\\
		&\leq  C\Bigg(\int_{r}^t |f_A((\Gamma Y_A^n)(h^n(y)),(\Gamma Y_P^n)(h^n(y)))|^pdy \nonumber \\&+ \int_{r}^t |g_A((\Gamma Y_A^n)(h^n(y)),(\Gamma Y_P^n)(h^n(y)))|^pdW_A(y)\nonumber\\ &+ \int_{r}^t |f_P((\Gamma Y_A^n)(h^n(y)),(\Gamma Y_P^n)(h^n(y)))|^pdy\nonumber \\&+\int_r^t \left|g_P((\Gamma Y_A^n)(h^n(y)),(\Gamma Y_P^n)(h^n(y)))\right|^pdW_P(y) \nonumber \\&+  \int_{r}^t \Bigg|\rho_A((\Gamma Y_A^n)(h^n(y)))\Bigg|^p dJ_A(y) \nonumber \\&+\int_r^t \Bigg|\rho_P((\Gamma Y_P^n)(h^n(y)))\Bigg|^p dJ_P(y)\Bigg).
	\end{align}
	Taking the expectation on \eqref{modulus2}, we obtain the following estimate
	\begin{align}\label{e1}
		E[|Y^n(t) - Y^n(r)|^p] \leq C(t-r)^p + CE\left[ \left|\int_r^t \sigma((\Gamma Y^n)(h^n(y)))dW(y)\right|^p\right].
	\end{align}
	Applying the Burkholder-Davis-Gundy's inequality to the second term of the right hand side of \eqref{e1}, we obtain
	\begin{align}
		E\left[ \left|\int_r^t \sigma((\Gamma Y^n)(h^n(y)))dW(y)\right|^p\right] \leq C E\left[ \left(\int_r^t dy \right)^{p/2}\right] \leq C(t-r)^{p/2}.
	\end{align}
	On the other hand, if $\alpha < \frac{1}{2}$, then
	\begin{align}
		\int_0^T \int_0^T \frac{1}{|t-r|^{1+(\alpha - \frac{1}{2}) p}} dt dr < \infty.
	\end{align}
	Consequently, we can pick $\alpha < \frac{1}{2}$. Taking now $p>2$ together with the constraint $\alpha p >1$, we can find $M_2 > 0$ such that 
	\begin{align}
		P([Y^n(t)]_{W^{\alpha, p}(S;\mathbb{R}^2 \times \mathbb{R}^2)} > M_2)  < \frac{\varepsilon}{2}.
	\end{align}
	This argument completes the proof of the Lemma.
\end{proof}

From Lemma \ref{tightness}, we have obtained that the sequence $\{\mathcal{P}^n\}$ is tight in $C(\bar{S};\mathbb{R}^2 \times \mathbb{R}^2)$. Applying the Prokhorov's Theorem, there are subsequences $\{\mathcal{P}^{n_k}\}$ which converge weakly to some $\mathcal{P}(Y(t))$ as $n \to \infty$. For simplicity of the notation, we denote these subsequences by $\{\mathcal{P}^{n}\}$. This means that we have $\{\mathcal{P}^n\}$ converging weakly to some probability measure $\mathcal{P}$ on Borel sets in $C(\bar{S};\mathbb{R}^2 \times \mathbb{R}^2)$. 

Since we have that $\mathcal{P}^n(Y^n(t))$ converges weakly to $\mathcal{P}(Y(t))$ as $n \to \infty$, by using the \textquotedblleft Skorokhod Representation Theorem\textquotedblright, there exists a probability space $(\widetilde{\Omega}, \tilde{\mathcal{F}},\tilde{P})$ with the filtration $\{\tilde{\mathcal{F}}_t\}$ and $\tilde{Y}^n(t)$, $\tilde{Y}(t)$ belonging to $C(\bar{S}; \mathbb{R}^2 \times \mathbb{R}^2)$ with $n\in \mathbb{N}$, such that $\mathcal{P}(\tilde{Y}) = \mathcal{P}(Y)$, $\mathcal{P}(\tilde{Y}^n(t)) = \mathcal{P}(Y^n(t))$, and $\tilde{Y}^n(t) \to \tilde{Y}(t)$ as $n \to \infty$, $\tilde{P}-$a.s.
Moreover, let $(\tilde{X}^n(t), \tilde{\Phi}^n(t))$ and $(\tilde{X}(t), \tilde{\Phi}(t))$ be the solutions of the following Skorokhod equations
\begin{align}\label{sols}
	\tilde{X}^n(t) &= \tilde{Y}^n(t) + \tilde{\Phi}^n(t), \nonumber\\
	\tilde{X}(t) &= \tilde{Y}(t) + \tilde{\Phi}(t),
\end{align}
respectively.  Then the continuity result in Theorem \ref{deterministic-theo} implies that the sequence $(\tilde{X}^n(t), \tilde{\Phi}^n(t))$ converges to $(\tilde{X}(t), \tilde{\Phi}(t)) \in C(\bar{S}; \bar{D}\times \bar{D}) \times C(\bar{S})$ uniformly in $t\in \bar{S}$, $\tilde{P}-$a.s as $n \rightarrow \infty$. Hence, we  still need to prove that $\tilde{Y}^n(t)$ converges to $\tilde{Y}(t)$ in some sense, where we denote
\begin{align}
	\tilde{Y}^n(t) := \tilde{X}_0 + \int_{0}^tf(\tilde{X}^n(h_n(y))dy + \int_0^t g(\tilde{X}^n(h_n(y))d\tilde{W}(y) + \int_{0}^{t}\rho(X^n(y))dJ(y),
\end{align}
and 
\begin{align}
	\tilde{Y}(t) := \tilde{X}_0 + \int_{0}^tf(\tilde{X}^n(y))dy + \int_0^t g(\tilde{X}^n(y))d\tilde{W}(y) +  \int_{0}^{t}\rho(X(y))dJ(y).
\end{align}
To complete the proof of the existence of solutions to the problem \eqref{main-skorohod}-\eqref{initial-cond} in the sense of Definition \ref{def}, we consider the following Lemma.
\begin{lemma}\label{solu-theo}
	The pair $(\tilde{X}(t), \tilde{\Phi}(t)) \in C(\bar{S}; \bar{D}\times \bar{D}) \times C(\bar{S})$ cf. \eqref{sols} is a solution of the Skorokhod-type system
	\begin{align}
		\tilde{X}(t) = \tilde{Y}(t) + \tilde{\Phi}(t),
	\end{align}
	where  
	\begin{align}
		\tilde{Y}(t) := \tilde{X}_0 + \int_{0}^tf(\tilde{X}^n(y))dy + \int_0^t g(\tilde{X}^n(y))d\tilde{W}(y),
	\end{align} 
	and $\tilde{X}_0 \in \bar{D}$.
\end{lemma}
\begin{proof}
	We consider the term $\int_0^t g(\tilde{X}^n(h_n(y)))d\tilde{W}(y)$ with the step process \\$g(\tilde{X}^n(h_n(y)))$. Approximating this stochastic integral by Riemann-Stieltjes sums (see e.g. \cite{Evans2013}), it yields 
	\begin{align}\label{reimann-app}
		\int_{0}^tg(\tilde{X}^n(h_n(y)))d\tilde{W}(y) = \sum_{k=0}^{n-1} g(\tilde{X}^n(h_n(t)))(W(t_{k+1}^n) - W(t_k^n)).
	\end{align}
	By taking the limit $n \to \infty$ in \eqref{reimann-app}, this gives
	\begin{align}\label{reimann}
		\lim_{n\to \infty}\int_{0}^tg(\tilde{X}^n(h_n(y)))d\tilde{W}(y) = \lim_{n\to \infty}\sum_{k=0}^{n-1} g(\tilde{X}^n(h_n(t)))(W(t_{k+1}^n) - W(t_k^n)) \nonumber\\=\sum_{k=0}^{n-1} g(\tilde{X}(t))(W(t_{k+1}) - W(t_k))  = \int_0^t g(\tilde{X}(y))d\tilde{W}(y).
	\end{align}
	By using the fact that $(\tilde{X}^n(t), \tilde{\Phi}^n(t))$ converges to $(\tilde{X}(t), \tilde{\Phi}(t))\in C(\bar{S}; \bar{D} \times \bar{D})\times C(\bar{S})$ uniformly in $t\in [0,T]$ $\tilde{P}-$a.s as $n \rightarrow \infty$ together with \eqref{reimann}, we obtain that
	\begin{align}
		\tilde{X}^n(t) &= \tilde{X}_0 + \int_{0}^tf(\tilde{X}^n(h^n(y)))dy + \int_0^t g(\tilde{X}^n(h^n(y)))d\tilde{W}(y) \nonumber\\&+\int_{0}^{t}\rho(X^n(h^n(y)))dJ(y) +  \tilde{\Phi}^n(t) \nonumber\\
		&= \tilde{Y}^n(t) + \tilde{\Phi}^n(t)
	\end{align}
	converges to 
	\begin{align}
		\tilde{X}(t) &= \tilde{X}_0 + \int_{0}^tf(\tilde{X}(y))dy + \int_0^t g(\tilde{X}(y))d\tilde{W}(y) + \int_{0}^{t}\rho(X(y))dJ(y) + \tilde{\Phi}(t)\nonumber\\
		&=\tilde{Y}(t) + \tilde{\Phi}(t), \quad \tilde{P}-\text{a.s as } n\to \infty.
	\end{align}
\end{proof}	
\subsubsection{Proof of the uniqueness}\label{uniqueness-proof}

We take $X(t), X'(t) \in C(\bar{S}; \bar{D} \times \bar{D})$ as two solutions to \eqref{main-skorohod}-\eqref{initial-cond} with the same initial values $X(0) = X'(0)$. 

Moreover, suppose that the supports of $b$ and $\sigma$ are included in the same ball $B(x_0,\delta)$ for some $x_0 \in \partial D$. Let us recall the assumption $(\text{A}_6)$, where $D$ satisfies the following condition: There exists a positive number $\eta$ such that for each $x_0 \in \partial D$ we can find $h\in C^2(\mathbb{R}^2 )$ satisfying $$\langle y-x, \textbf{n}\rangle + \frac{1}{\nu}\langle\nabla h(x), \textbf{n}\rangle|y-x|^2 \geq 0$$
for any $x \in B(x_0,\delta') \cap \partial D, y \in B(x_0, \delta'') \cap \partial \bar{D}$ and $\textbf{n} \in \mathcal{N}_{x}$. Using the proof idea of Lemma $5.3$ in \cite{Saisho1987}, we consider $f(x) = \langle \textbf{l}, x-x_0\rangle$. Then, we have 
\begin{align}\label{cond-1}
	\langle X(s) - X'(s), d\Phi(s) - d\Phi'(s)\rangle - \frac{1}{\eta}|X(s) - X'(s)|^2\langle\textbf{l}, d\Phi(s) - d\Phi'(s)\rangle \nonumber\\
	=-(\langle X(s) - X'(s), d\Phi(s)\rangle +  \frac{1}{\eta}|X(s) - X'(s)|^2\langle\textbf{l}, d\Phi(s)\rangle) \nonumber\\- (\langle X(s) - X'(s), d\Phi'(s)\rangle +  \frac{1}{\eta}|X(s) - X'(s)|^2\langle\textbf{l}, d\Phi'(s)\rangle).
\end{align}
Moreover, using the assumption $(\text{A}_6)$, we have the following estimates
\begin{align}\label{pre-1}
	\langle X(s) - X'(s), d\Phi(s)\rangle +  \frac{1}{\eta}|X(s) - X'(s)|^2\langle\textbf{l}, d\Phi(s)\rangle \geq 0 
\end{align}
and 
\begin{align}\label{pre-2}
	\langle X(s) - X'(s), d\Phi'(s)\rangle +  \frac{1}{\eta}|X(s) - X'(s)|^2\langle\textbf{l}, d\Phi'(s)\rangle \geq 0.
\end{align}
Combinning \eqref{cond-1}-\eqref{pre-2}, we obtain the following estimate
\begin{align}
	\langle X(s) - X'(s), d\Phi(s) - d\Phi'(s)\rangle - \frac{1}{\eta}|X(s) - X'(s)|^2\langle\textbf{l}, d\Phi(s) - d\Phi'(s)\rangle \leq 0,
\end{align}
where $\textbf{l}$ is the unit vector appearing in Condition $(\text{A}_5)$.

Using similar ideas as in \cite{Lions1984} (see Proposition $4.1$), we have the following estimate
\begin{align}\label{est-uni1}
	&|X(t) - X'(t)|^2\exp\left\{-\frac{1}{\eta}(\Phi(X(t)) - \Phi'(X(t)))\right\} \leq \nonumber\\&2\Bigg(\exp\left\{-\frac{1}{\eta}(\Phi(X(y)) - \Phi'(X(y)))\right\}\int_0^t(b(X(y)) - b(X'(y)))dy \nonumber\\&+ \exp\left\{-\frac{1}{\eta}(\Phi(X(y)) - \Phi'(X(y)))\right\}\int_0^t(\sigma(X_y(y)) - \sigma(X(y))) dW(y) \Bigg)^2 \nonumber\\&+ \exp\left\{-\frac{1}{\eta}(\Phi(X(y)) - \Phi'(X(y)))\right\}\nonumber\\&\times\int_0^t\left(2\langle X(y) - X'(y),\textbf{l}\rangle - \frac{1}{\eta}|X(y) - X'(y)|^2\right)d\Phi(y)\nonumber\\
	&+\exp\left\{-\frac{1}{\eta}(\Phi(X(y)) - \Phi'(X(y)))\right\}\nonumber\\&\times\int_0^t\left(2\langle X(y) - X'(y),\textbf{l}\rangle - \frac{1}{\eta}|X(y) - X'(y)|^2\right)d\Phi'(y) \nonumber\\
	&+2\int_0^t\left|b(X(y)) - b(X'(y))\right|^2\exp\left\{-\frac{2}{\eta}(\Phi(X(y)) - \Phi'(X(y)))\right\}dy \nonumber\\&+ 2\int_0^t\left|\sigma(X(y)) - \sigma(X(y))\right|^2\exp\left\{-\frac{2}{\eta}(\Phi(X(y)) - \Phi'(X(y)))\right\} dy \nonumber\\
	&+ \int_0^t\left(2\langle X(y) - X'(y),\textbf{l}\rangle - \frac{1}{\eta}|X(y) - X'(y)|^2\right)\nonumber\\&\times\exp\left\{-\frac{1}{\eta}(\Phi(X(y)) - \Phi'(X(y)))\right\}d\Phi(y)\nonumber\\&+\int_0^t\left(2\langle X(y) - X'(y),\textbf{l}\rangle - \frac{1}{\eta}|X(y) - X'(y)|^2\right)\nonumber\\&\times\exp\left\{-\frac{1}{\eta}(\Phi(X(y)) - \Phi'(X(y)))\right\}d\Phi'(y).
\end{align}
On the other hand, taking the expectation on both sides of \eqref{est-uni1} and using the Lipschitz condidion to the first term of the right hand side together with \eqref{cond-1},  we are led to

\begin{align}
	E\left(|X(t) - X'(t)|^2\exp\left\{-\frac{1}{\eta}(\Phi(X(t)) - \Phi'(X(t)))\right\}\right) \leq \nonumber\\ C\int_0^tE\Big(|X(y) - X'(y)|^2\exp\left\{-\frac{2}{\eta}(\Phi(X(y)) - \Phi'(X(y)))\right\}\Big)dy.
\end{align}
This also implies that

\begin{align}
	E[|X(t) - X'(t)|^2] \leq C\int_0^tE[|X(y) - X'(y)|^2]dy.
\end{align}
Hence, $X(t)=X'(t)$ for all $t \in [0,T]$. Then, the pathwise uniqueness of solutions to \eqref{main-skorohod}-\eqref{initial-cond} holds true.  Moreover, the pathwise uniqueness implies the uniqueness of strong solutions (see in \cite{Ikeda1981}, Theorem IV-1.1). On the other hand, combining the Lemma \ref{solu-theo} and the result provided in \cite{Yamada1971}, the system of reflected SDEs \eqref{main-skorohod}-\eqref{initial-cond} admits a unique strong solution $(X(t), \Phi(t)) \in C(\bar{S}; \bar{D}\times \bar{D})\times C(\bar{S})$.

\subsubsection{Proof of the dependence on the parameters}\label{stability-proof}

\begin{proof}
	Let us recall our system of SDEs from \eqref{main-skorohod2}, Section \ref{wellpossedness}:
	
	\begin{align}
		\begin{cases}
			dX^n(t) = f(X^n(t)) dt +g (X_t^n(t))dW(t) + \int_{\mathbb{R}}\rho(X^n(t),y) \nu (dy,dt) + d\Phi^n(t),\\
			X^n(0) = X_0^n \in \bar{D} \text{ for } n \geq 1.
		\end{cases}
	\end{align}
	For its solution, we have
	
	
	\begin{align}\label{sta-1}
		X_t^n(t) = X_0^n + \int_0^t f(X_z^n(z))dz + \int_0^t  g (X^n(z))dB(z) + \Phi^n(t).
	\end{align}
	Let us consider the following equation
	
	\begin{align}
		X^n(t) - X_t(t) &= X_0^n - X_0 + \int_0^t f(X^n(y))dy - \int_0^t f(X(y))dy \nonumber\\&+ \int_0^t g (X^n(y))dW(y) - \int_0^t g (X(y))dW(y) \nonumber\\&+\int_0^t \rho(X^n(y))dJ(y)+ \int_0^t \rho(X(y))dJ(y)+ \Phi^n(t) - \Phi(t).
	\end{align}
	Since
	$(a + b+ c + d+ e)^2 \leq 5a^2 + 5b^2+ 5c^2+ 5d^2$ for any $a, b, c, d, e  \in \mathbb{R}$, we have the following estimate
	\begin{align}\label{sta-2}
		|X^n(t) - X(t)|^2 &\leq 5|X_0^n - X_0|^2 + 5\left|\int_0^t f(X^n(y))dy - \int_0^t f(X(y))dy\right|^2 \nonumber\\&+ 5\left|\int_0^t g (X^n(y))dW(y) - \int_0^t g (X(y))dW(y)\right|^2 \nonumber\\&+ 5\left|\int_0^t \rho(X^n(y))dJ(y) - \int_0^t \rho(X(y))dJ(y) \right|^2+ 5|\Phi^n(t) - \Phi(t)|^2.
	\end{align} 
	Taking the expectation on both sides of  \eqref{sta-2}, we have
	\begin{align}\label{stabi-1}
		E(|X^n(t) - X(t)|^2) &\leq 5E(|X_0^n - X_0|^2) \nonumber\\&+ 5E\left(\left|\int_0^t f(X^n(y))dy - \int_0^t f(X(y))dy\right|^2\right) \nonumber\\&+ 5E\left(\left|\int_0^t g(X^n(y))dW(y) - \int_0^t \sigma (X(y))dW(y)\right|^2\right) \nonumber\\&+ 5E \left(\left|\int_0^t \rho(X^n(y))dJ(y) - \int_0^t \rho(X(y))dJ(y) \right|^2\right)\nonumber\\&+5E\left(|\Phi^n(t) - \Phi(t)|^2\right).
	\end{align} 
	To begin with, we consider the second and the third terms of the right-hand side of \eqref{stabi-1}. Using Cauchy-Schwarz inequality together with the assumption that $f, g$ are Lipschitz functions, we are led to
	
	\begin{align}\label{sta-1-1}
		E\left(\left|\int_0^t f(X^n(y))dy - \int_0^t f(X(y))dy\right|^2\right) &\leq CE\left(\int_0^t |f(X^n(y)) - f(X(y))|^2dy\right)\nonumber\\
		&\leq C\int_0^tE(|X^n(y) - X(y)|^2)dy
	\end{align}
	and 
{\small	\begin{align}\label{sta-1-2}
		E\left(\left|\int_0^t g(X^n(y))dW(y) - \int_0^t g (X(y))dW(y)\right|^2\right) &= E\left(\int_0^t |g(X^n(y)) - g (X(y))|^2dy\right)\nonumber\\
		&\leq C\int_0^t E(|X^n(y) - X(y)|^2)dy.
	\end{align}}
	Moreover, using \eqref{reflection-func}, it yields
	\begin{align}
		|\Phi^n(t) - \Phi(t)| &\leq 2|X_0^n - X_0| + 2\left|\int_0^t f(X^n(y))dy - \int_0^t f(X(y))dy\right| \nonumber\\&+ 2\left|\int_0^t g (X^n(y))dW(y) - \int_0^t g (X(y))dB(y)\right|\nonumber\\&+ 2\left|\int_0^t \rho(X^n(y))dJ(y) - \int_0^t \rho(X(y))dJ(y) \right|.
	\end{align}
	Since $(a+b+c + d)^2 \leq 4a^2+4b^2+4c^2 + 4d^2$ for all $a, b, c, d \in \mathbb{R}$, we have
	\begin{align}\label{sta-3}
		|\Phi^n(t) - \Phi(t)|^2 &\leq 8|X_0^n - X_0|^2 + 8\left|\int_0^t f(X^n(y)) - \int_0^t f(X(y))dy\right|^2 \nonumber\\&+ 8\left|\int_0^t g (X^n(y))dW(y) - \int_0^t g (X(y))dW(y)\right|^2 \nonumber\\&+ 8\left|\int_0^t \rho(X^n(y))dJ(y) - \int_0^t \rho(X(y))dJ(y) \right|^2.
	\end{align}
	Taking the expectation on both sides of \eqref{sta-3}, we are led to
	\begin{align}\label{stabi-es1}
		E(|\Phi_t^n(t) - \Phi_t(t)|^2) &\leq 8E(|X_0^n - X_0|^2) + 8E\left(\left|\int_0^t f(X^n(y)) - \int_0^t f(X(y))dy\right|^2\right) \nonumber\\&+ 8E\left(\left|\int_0^t g (X^n(y))dW(y) - \int_0^t g (X(y))dW(y)\right|^2\right)\nonumber\\&+ 8E\left(\left|\int_0^t \rho(X^n(y))dJ(y) - \int_0^t \rho(X(y))dJ(y) \right|^2\right).
	\end{align}
	By applying Cauchy-Schwarz's inequality to the second and third terms of the right-hand side of \eqref{stabi-es1}, we have the following estimate
	\begin{align}
		E(|\Phi^n(t) - \Phi(t)|^2) &\leq 8E(|X_0^n - X_0|^2) + 
		CE\left(\int_0^t |f(X^n(y))dy - f(X(y))|^2dy\right) 
		\nonumber\\&
		+ 8E\left(\int_0^t \left|g(X^n(y))dW(y) - g(X(y))\right|^2dy\right) 	\nonumber\\& 8E\left(\int_0^t \left|\rho(X^n(y)) - \rho(X(y))\right|^2dJ(y)\right).
	\end{align}
	Using again the assumption that $b, \sigma$ are Lipschitz functions, we get next the following estimate
	\begin{align}\label{sta-phi}
		E(|\Phi^n(t) - \Phi(t)|^2) \leq 8E(|X_0^n - X_0|^2) + C\int_0^t E(|X^n(y) - X(y)|^2)dy \nonumber\\+ C\int_0^t E(|X^n(y) - X(y)|^2)dy + C\int_0^t E(|X^n(y) - X(y)|^2dJ(y)).
	\end{align}
	Then, by using \eqref{sta-1-1}, \eqref{sta-1-2} and \eqref{sta-phi}, the inequality \eqref{sta-1} becomes
	\begin{align}\label{bfgronwall}
		E(|X^n(t) - X(t)|^2) &\leq CE(|X_0^n - X_0|^2) + C\int_0^t E(|X^n(y) - X(y)|^2)dy  \nonumber\\&+ C\int_0^t E(|X^n(y) - X(y)|^2dJ(y)),
	\end{align}
	for $0 \leq t \leq T$.
	Applying Gronwall's inequality to \eqref{bfgronwall} yields
	\begin{align}\label{gronwall}
		E(|X^n(t) - X(t)|^2) \leq CE(|X_0^n - X_0|^2).
	\end{align}
	Moreover, we have that
	\begin{align}\label{sta-4}
		\max_{0\leq t \leq T}|X^n(t) - X(t)|^2 &\leq 5|X_0^n - X_0|^2 + C\int_0^T |X^n(t) - X(t)|^2dt \nonumber\\&+ 5\max_{0 \leq t \leq T}\left|\int_0^T g (X^n(t)) - g (X(t))dB(t)\right|^2 \nonumber\\&+ C\max_{0 \leq t \leq T}\int_0^T \left|X^n(y) - X(y)\right|^2dJ(y) \nonumber\\&+  5\max_{0 \leq t \leq T} |\Phi_t^n(t) - \Phi_t(t)|^2.
	\end{align}
	After taking the expectation on both sides of \eqref{sta-4}, we apply the martingale inequality to the third term on the right-hand side of the resulting inequality, which reads
	\begin{align}
		E&\left(\max_{0\leq t \leq T}|X_t^n(t) - X_t(t)|^2\right) \leq 5E(|X_0^n - X_0|^2) + C\int_0^T E(|X_t^n(t) - X_t(t)|^2)dt \nonumber\\&+ C\int_0^T E(|X_t^n(t) - X_t(t)|^2)dt + 5E\left(\max_{0 \leq t \leq T} |\Phi_t^n(t) - \Phi_t(t)|^2\right) \nonumber\\&+ C\int_0^t E(|X^n(y) - X(y)|^2dJ(y)) \nonumber\\
		&\leq CE(|X_0^n - X_0|^2) + C\int_0^T E(|X_t^n(t) - X_t(t)|^2)dt \nonumber\\&+ C\int_0^t E(|X^n(y) - X(y)|^2dJ(y)). 
	\end{align}
	Finally, using \eqref{sta-phi} and \eqref{gronwall}, we obtain the desired estimate: 
	\begin{align}
		E\left(\max_{0\leq t \leq T}|X^n(t) - X(t)|^2\right) \leq CE(|X_0^n - X_0|^2).
	\end{align}
	By using the fact that 
	$\lim_{n\to \infty} E(|X_0^n - X_0|^2) = 0$,
	we obtain the following estimate
	\begin{align}
		\lim_{n \to \infty} E\left(\max_{0\leq t \leq T}|X^n(t) - X(t)|^2\right) = 0.
	\end{align}
\end{proof}

\section{Applications of coupled stochastic processes in bounded domains}\label{applications}

In general, in the study of biological systems, the descriptions of individual cells may be appropriate for primitive systems. However, to model reliably living systems with complex biological interactions, a large number of cells needs to be accounted for. For instance, the human brain consists of approximately 1011 neurons and is connected to 104 other neurons \cite{Dowling2001}. To better understand the resulting neural activity requires appropriate models that can track the average firing rate across many areas of a neuronal network. Therefore, from a large population of densely coupled neurons, Wilson and Cowan \cite{Wilson1972,Kilpatrick2014} have derived an effective system for the proportion of cells in a population that are active per unit time. In this section, we consider an application of our active-passive population dynamics in the model of synaptically coupled excitatory and inhibitory neurons in the
neocortex. In particular, we study a system of stochastic Wilson-Cowan-type equations with reflection and possible jump-like exit from a boundary
that allows us to model the dynamics of two interacting populations of excitatory and inhibitory neurons (see, e.g., Fig. \ref{fig:0-1}). Let us recall the following Wilson-Cowan system, considering a 2-dimensional dynamic case (see, e.g. \cite{Wilson1972})
\begin{align}\label{main_WC}
	\begin{cases}
		\tau_E\frac{dr_E}{dt} = -r_E(t) + (1-\delta_Er_E(t))F_E(w_{EE}r_E(t) - w_{EI}r_I(t) + I_E^{\text{ext}},\theta_E,a_E),\\
		\tau_I\frac{dr_I}{dt} = -r_I(t) + (1-\delta_Ir_I(t))F_I(w_{IE}r_E(t) - w_{II}r_I(t) + I_I^{\text{ext}},\theta_I,a_I),
	\end{cases}
\end{align}
where $r_E(t)$ and $r_I(t)$ are the proportions of excitatory and inhibitory cells firing per unit time at the instant $t$, respectively. Here, $w_{EE}$ and $w_{II}$ represent the strengths of connection between excitatory and inhibitory cells, respectively, while $w_{EI}$ describes the strength of connection from excitatory cells to inhibitory cells and $w_{IE}$ denotes the strength of connection from inhibitory cells to excitatory cells. Moreover, $\tau_E$ and $\tau_I$ denote the refractory periods of excitatory and inhibitory cells after a trigger, respectively, while $r_E$ and $r_I$ are the absolute refractory periods, $\theta_E$ and $\theta_I$ are the threshold of the excitatory and inhibitory populations. We also assume that $r_E=0$ and $r_I = 0$ correspond to a low-activity resting states of excitatory and inhibitory cells. In \eqref{main_WC}, functions $F_E$ and $F_I$ represent the nonlinearities typically chosen to be sigmoidal defined as (see, e.g. \cite{Wilson1972})
\begin{align}
	F(x,\theta,a) := \frac{1}{1+\exp[-a(x-\theta)]} - \frac{1}{1+\exp(a\theta)}. 
\end{align}

\begin{figure}[h!]
	\centering
	\includegraphics[width=0.8\textwidth]{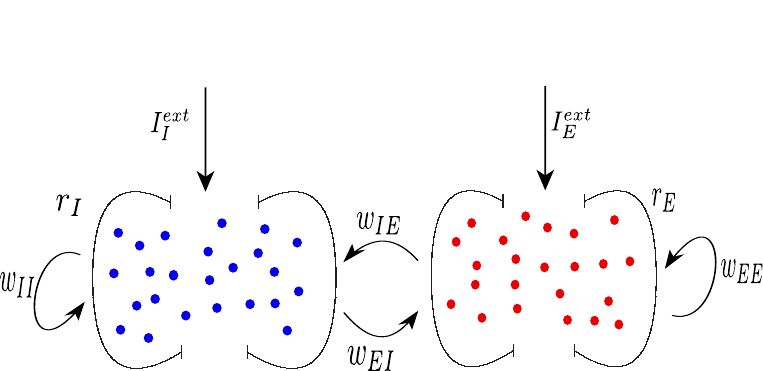} 
	\caption{[Color online] Sketch of networks of interacting excitatory and inhibitory populations.}
	\label{fig:0-1}
\end{figure}

In general, an excitatory transmitter generates an electrical signal called an action potential in the receiving neuron, while an inhibitory transmitter prevents such electrical signals (see, e.g., \cite{Chan2019}). Hence, we assume that an excitatory population can be considered an active population, while an inhibitory population can be seen as a passive population. Neuron dynamics are intensively computed and often deal with many challenges from severe accuracy degradation if the input data is corrupted with noise. Furthermore,  the noise is normally assumed as purely diffusive noise, namely, as random fluctuations with continuous sample paths. However, such a description is incomplete due to the fact that the diffusive fluctuations are large and abrupt events appear at random times throughout the time series \cite{Melanson2019,Bressloff2011,Powanwe2021}. To get closer to the real scenarios in biological systems,  jump-diffusion stochastic processes provide a more appropriate framework to model these data.  
Using the descriptions provided in the previous sections, we focus on investigating the system of stochastic Wilson-Cowan-type equations with reflection and possible jump-like exit from a boundary, which reads:
\begin{align}\label{WC_2}
	\begin{cases}
		dr_E(t) &= \frac{1}{\tau_E}(- r_E(t) + (1-\delta_Er_E(t))\tilde{F}_E(w_{EE}r_E(t) - w_{EI}r_I(t) + I_E^{\text{ext}},\theta_E,a_E))dt\\&+ \frac{\sigma_E^{\text{ext}}}{\tau_E}(1-\delta_Er_E(t))dW_E(t)  + d\phi_E(t)+ \int_{\mathbb{R}}\rho_E(r_E(t))dJ_E(t),\\
		dr_I(t) &= \frac{1}{\tau_I}(- r_I(t) + (1-\delta_Ir_I(t))\tilde{F}_I(w_{IE}r_E(t) - w_{II}r_I + I_I^{\text{ext}},\theta_I,a_I))dt\\&+ \frac{\sigma_I^{\text{ext}}}{\tau_I}(1-\delta_Ir_I(t))dW_I(t) +d\phi_I(t) + \int_{\mathbb{R}}\rho_I(r_I(t))dJ_I(t),
	\end{cases}
\end{align}
In \eqref{WC_2}, we assume that $F_E(w_{EE}r_E - w_{EI}r_I + I_E^{\text{ext}},\theta_E,a_E) = \tilde{F}_E(w_{EE}r_E - w_{EI}r_I + I_E^{\text{ext}},\theta_E,a_E)) + \sigma_E^{\text{ext}}dW_E(t)$ and $F_I = (w_{IE}r_E - w_{II}r_I + I_I^{\text{ext}},\theta_I,a_I)) + \sigma_I^{\text{ext}}dW_I(t)$. Note that $f_{A}(X_{A}(t), X_{P}(t))$ and $f_{P}(X_{A}(t), X_{P}(t))$ (in \eqref{Active_population}) become $\frac{1}{\tau_E}(- r_E + (1-\delta_Er_E)\tilde{F}_E(w_{EE}r_E - w_{EI}r_I + I_E^{\text{ext}},\theta_E,a_E))$ and $\frac{1}{\tau_I}(- r_I + (1-\delta_Ir_I)\tilde{F}_I(w_{IE}r_E - w_{II}r_I + I_I^{\text{ext}},\theta_I,a_I))$, respectively. Similarly, $g_{A}(X_{A}(t), X_{P}(t))$ and  $g_{P}(X_{A}(t), X_{P}(t))$ \\(in \eqref{Active_population}) become $ \frac{\sigma_E^{\text{ext}}}{\tau_E}(1-\delta_Er_E)$ and $\frac{\sigma_I^{\text{ext}}}{\tau_I}(1-\delta_Ir_I)$, while $\rho_A, \rho_P, \phi_{A}, \phi_{P}$ can be considered as $\rho_E, \rho_I, \phi_{E}, \phi_I$, respectively.

The simulations presented in this section have been carried out by using by a discrete-time integration based on the Euler method inplemented in Python. In this section, for $t\geq 0$, we consider the case of $r_E(t),r_I(t) \geq 0$ with $L(0) = 0$ and
\begin{align}
	\phi_{E}(t) = \int_{0}^{t} \mathbbm{1}_{r_E(s) = 0}d\phi_{E}, \text{ and }
	\phi_{I}(t) = \int_{0}^{t} \mathbbm{1}_{r_I(s) = 0}d\phi_{I}.
\end{align}
This condition implies that the process can increase only when $r_E$ and $r_I$ hit 0 (see, e.g., \cite{Pilipenko2012}). In other words, this is the reflecting boundary condition at 0 in one-dimensional for each population in the domain $D = [0, \infty)$. Moreover, we use a set of compound Poission process as in \eqref{compoundPossion1} to describe the jump-like exit at the boundary. 

In the simulations, we choose the parameter set as follows: $r_E = r_I = 0.2$, $\tau_E = 1$ (ms), $\tau_I = 2$ (ms), $\theta_E = 2.8$, $\theta_I = 4$, $a_E=1.2$, $a_I = 1$, $w_{EE} = 12$, $w_{EI} = 4$, $w_{IE} = 13$, $w_{II} = 11$, $dt=0.1$ (ms). These parameters have also been used in \cite{Wilson1972}.

Let $\{V(t) : t \geq 0\}$ be an Ornstein-Uhlenbeck process (see, e.g., Fig. \ref{fig:1}) defined on $[r, \infty)$ with drift $(\mu - V(t)/\gamma)$ and constant diffusion parameter $\sigma$. Then, the process $\{V(t) : t \geq 0\}$ satisfies the following SDE:

\begin{align}\label{OU}
	\begin{cases}
		dV(t) = (\mu - \frac{1}{\gamma}V(t))dt + \sigma dW(t), \\
		V(0) \in [r, \infty),
	\end{cases}
\end{align}
where $W(t)$ denotes Gaussian white noise. 

The main representative numerical results of our analysis are shown in Fig. \ref{fig:3}, where we have plotted the population trajectories of excitatory and inhibitory populations.

\begin{figure}[h!]
	\centering
	\includegraphics[width=0.8\textwidth]{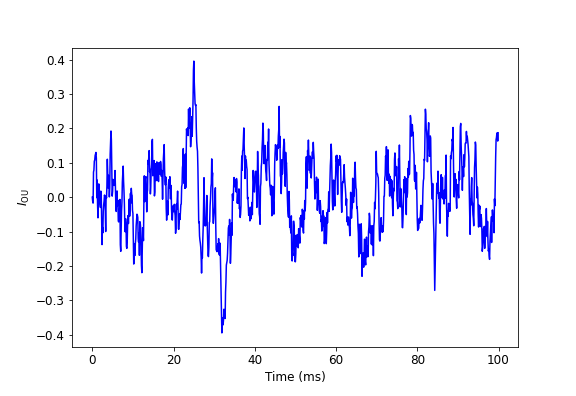} 
	\caption{[Color online]  Ornstein-Uhlenbeck input current profile satifies the Ornstein-Uhlenbeck process \eqref{OU}. Parameters: $\mu=0$, $\gamma=1$ and $\sigma=0.1$. }
	\label{fig:1}
\end{figure}

\begin{figure}[h!]
	\centering
	\begin{tabular}{ll}
		\includegraphics[width=0.45\textwidth]{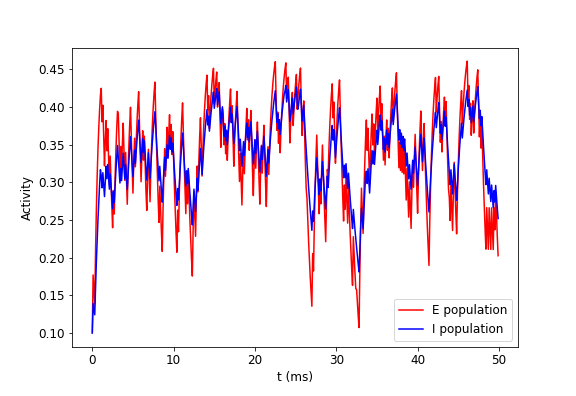} &\includegraphics[width=0.45\textwidth]{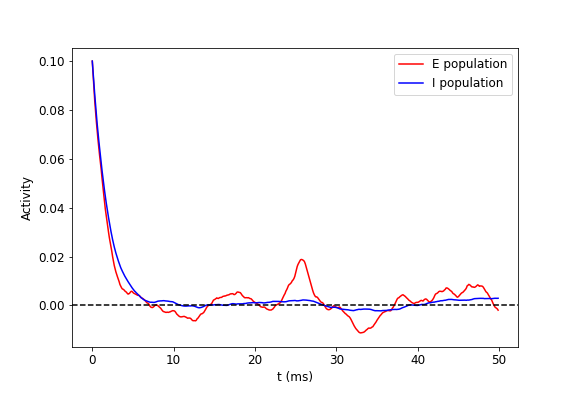} \\
		\includegraphics[width=0.45\textwidth]{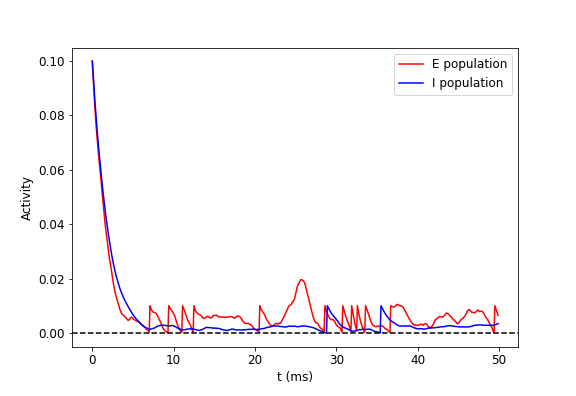} &\includegraphics[width=0.45\textwidth]{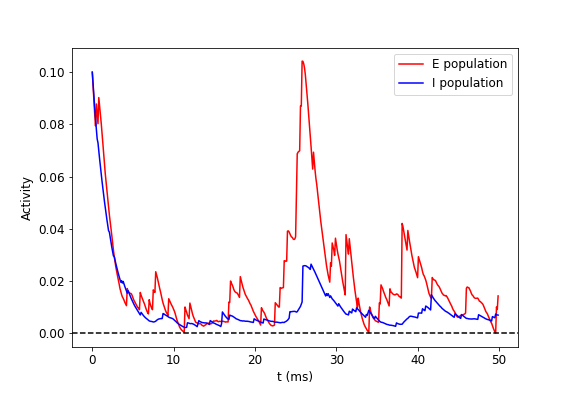}
	\end{tabular}
	\caption{[Color online] The population trajectories of excitatory (red color) and inhibitory populations (blue color) of the system \eqref{WC_2}.
		Top left: Gaussian white noise input current. Top right: Ornstein-Uhlenbeck input current (see, e.g., Fig. \ref{fig:1}). Bottom left: Ornstein-Uhlenbeck input current with reflecting boundary condition at 0. Bottom right: Ornstein-Uhlenbeck input current with reflecting boundary condition at 0 and with jumps. Parameters: $\mu=0$, $\gamma=1$ and $\sigma=0.1$.
	}\label{fig:3}	
\end{figure}

In particular,  in the top left panel of Fig. \ref{fig:3}, we have plotted the population trajectories of excitatory and inhibitory populations under only Gaussian white noise input current. We see that there are fluctuations in the time evolution of the proportions of both excitatory and inhibitory cells firing. However, in the top right panel of Fig. \ref{fig:3}, when we replace the Gaussian white noise current with the Ornstein-Uhlenbeck input current presented in Fig. \ref{fig:1}, we observe that the firing activity of excitatory and inhibitory cells fluctuates to values less than zero. Therefore, we add the reflecting factor to our system with the Ornstein-Uhlenbeck input current and we see that the firing activity increases to values larger than 0 (the resting states of excitatory and inhibitory cells) in the bottom right panel of Fig. \ref{fig:3}. Moreover, in the presence of Poisson jumps and the reflecting factor in our system with Ornstein-Uhlenbeck input current, the firing activities of excitatory and inhibitory cells increase dramatically, as seen in the top right panel of Fig. \ref{fig:3}, compared to the case presented in the top left panel of the same figure. Specifically, the firing rate of excitatory cells increases to 0.1 at $t=27$ (ms) and the firing rate of the inhibitory cells increases to 0.02 at the same time. However, the firing rate of inhibitory cells is less than the firing rate of excitatory cells in the presence of jumps in the system. 

Additionally, we have observed that in the presence of Ornstein-Uhlenbeck input current, reflecting factors together with the Poisson jumps increase the firing activity of excitatory and inhibitory populations. This effect may lead to an improvement in the response of neurons to each stimulus in neuronal systems. 

\section{Conclusions}

We have proposed a model based on a coupled system of reflecting Skorokhod-type stochastic differential equations with jumps. We have analyzed the well-posedness of such systems.
In particular, using compactness methods and Skorokhod's representation of solutions to SDEs with the jump-like exit from a boundary, we have shown the existence and uniqueness of the solutions of these systems. Additionally, the structure of the Skorokhod problem allowed us to prove also the solution dependence on the parameters of our system. On the other hand, the mathematical setting of our system of SDEs has demonstrated a new point of view useful for the field of modelling and simulations of population dynamics. We provided details of the models along with representative numerical
examples and discussed the applications of our population dynamics in applications to neuronal dynamics. In particular, we have considered a system of stochastic Wilson-Cowan-type equations with reflection and possible jump-like exit from a boundary. Our numerical results have shown that the presence of Ornstein-Uhlenbeck input current, reflecting factors together with the Poisson jumps, strongly affects the firing activities of excitatory and inhibitory populations in a neuronal system. 

\appendix
\section{Technical preliminaries for Section \ref{wellpossedness}}\label{appendix}

In the proof of existence in Section \ref{wellpossedness}, we used compactness arguments. Here, we provide necessary details for such arguments to hold. We recall the classical Ascoli-Arzel\`{a} Theorem \cite{Rudin1953}:\\
A family of functions $U \subset C(\bar{S};\mathbb{R}^d)$ is relatively compact (with respect to the uniform topology) if
\begin{enumerate}
	\item[i.] for every $t \in \bar{S}$, the set $\{h(t);h \in U\}$ is bounded.
	\item[ii.] for every $\varepsilon>0$ and $t,s \in \bar{S}$, there is $\bar{\delta} > 0$ such that
	\begin{eqnarray}
		|h(t)-h(s)| \leq \varepsilon,
	\end{eqnarray}
	whenever $|t - s| \leq \bar{\delta}$ for all $h \subset U$.
\end{enumerate}
For a function $h: \bar{S} \to \mathbb{R}^d$,  we introduce the definition of H\"{o}lder seminorms as
\begin{eqnarray}\label{holder_norm}
	[h]_{C^\alpha(\bar{S}; \mathbb{R}^d)} = \sup_{t\neq s; t,s \in \bar{S}}\frac{|h(t)-h(s)|}{|t-s|^{\alpha}}, 
\end{eqnarray}
for $\alpha \in (0,1)$ and the supremum norm as
\begin{eqnarray}\label{infty_norm}
	\|h\|_{L^\infty(S;\mathbb{R}^d)} = \mathrm{ess}\sup_{t \in \bar{S}}|h(t)|.
\end{eqnarray}
We refer to \cite{Adams03} and \cite{Gilbarg1977} for more details on the used function spaces.

In fact, the simple sufficient conditions for $\text{i}.$ and $\text{ii}.$ are 
\begin{enumerate}
	\item[i'.] there is $M_1>0$ such that $\|h\|_{L^\infty(S;\mathbb{R}^d)}\leq M_1$ for all $h \in U$,
	\item[ii'.] for some $\alpha \in (0,1)$, there is an $M_2>0$ such that $[h]_{C^\alpha(\bar{S}; \mathbb{R}^d)} \leq M_2$ for all $h \in U$.
\end{enumerate}
Hence, we have the sets 
\begin{eqnarray}\label{relativelycompactKMP}
	K_{M_1M_2} := U = \left\{h\in C(\bar{S};\mathbb{R}^d); \|h\|_{L^\infty(S;\mathbb{R}^d)}\leq M_1, [h]_{C^\alpha(\bar{S}; \mathbb{R}^d)} \leq M_2\right\}
\end{eqnarray}
are relatively compact in $C(\bar{S}; \mathbb{R}^d)$.

For $\alpha \in (0,1)$, $T>0$ and $p>1$, the space 
$W^{\alpha,p}(S;\mathbb{R}^d)$ is defined as the set of all $h \in L^p(S;\mathbb{R}^d)$ such that
\begin{eqnarray}
	[h]_{W^{\alpha,p}(S;\mathbb{R}^d)}:= \int_{0}^{T}\int_{0}^{T}\frac{|h(t) - h(s)|^p}{|t-s|^{1 + \alpha p}}dtds < \infty. \nonumber
\end{eqnarray}
This space is endowed with the norm
\begin{eqnarray}
	\|h\|_{W^{\alpha,p}(S;\mathbb{R}^d)} = \|h\|_{L^p(S;\mathbb{R}^d)} + [h]_{W^{\alpha,p}(S;\mathbb{R}^d)}.\nonumber
\end{eqnarray}

Moreover, we have the following embedding
\begin{eqnarray}
	W^{\alpha,p}(S; \mathbb{R}^d) \subset C^\gamma(\bar{S};\mathbb{R}^d) \quad \textrm{ for } \alpha p - \gamma >1\nonumber
\end{eqnarray}
and $[h]_{C^\gamma(\bar{S}; \mathbb{R}^d)} \leq C_{\gamma,\alpha,p}\|h\|_{W^{\alpha,p}(S; \mathbb{R}^d)}$ (see e.g. in Theorem $6$, Chapter $5$ in \cite{Evans1998}). Relying on the Ascoli-Arzel\`{a} Theorem, we have the following situation:
\begin{enumerate}
	\item[ii''.] for some $\alpha \in (0,1)$ and $p>1$ with $\alpha p>1$, there is $M_2>0$ such that $[h]_{W^{\alpha,p}(S;\mathbb{R}^d)} \leq M_2$ for all $h \in U' \subset C(\bar{S};\mathbb{R}^d)$. 
\end{enumerate}
If $\textrm{i'}$ and $\textrm{ii''}$ hold, then the set
\begin{eqnarray}\label{relativelycompactKMP_2}
	K'_{M_1M_2} :=U'= \left\{h \in C(\bar{S}; \mathbb{R}^d); \|h\|_{L^\infty(S;\mathbb{R}^d)} \leq M_1, [h]_{W^{\alpha,p}(S;\mathbb{R}^d)} \leq M_2\right\}
\end{eqnarray}
is relatively compact in $C(\bar{S}; \mathbb{R}^d)$, if $\alpha p>1$ (see e.g. \cite{Flandoli95}, \cite{Colangeli2019}). In the main part of the manuscript, this result was formulated in formula \eqref{km1m2}.

\section*{Acknowledgments}
TKTT and RM are grateful to the NSERC and the CRC Program for their
support. RM is also acknowledging support of the BERC 2022-2025 program and Spanish Ministry of Science, Innovation and Universities through the Agencia Estatal de Investigacion (AEI) BCAM Severo Ochoa excellence accreditation SEV-2017-0718 and the Basque Government fund AI in BCAM EXP. 2019/00432.
TKTT and AM thank O. M. Richardson (Karlstad), E.N.M. Cirillo (Rome) and M. Colangeli (L'Aquila) for very fruitful discussions on the topic of active-passive population dynamics through heterogeneous environments.

\bibliographystyle{siamplain}
\bibliography{mybibfile}

\begin{thebibliography}{10}

\bibitem{Adams03}
{\sc R.~A. Adams and J.~J. Fournier}, {\em Sobolev {S}paces}, vol.~140,
  Academic Press, 2003.

\bibitem{Bressloff2011}
{\sc P.~C. Bressloff and Y.~M. Lai}, {\em Stochastic synchronization of
  neuronal populations with intrinsic and extrinsic noise}, The Journal of
  Mathematical Neuroscience, 1 (2011), pp.~1--28.

\bibitem{Chan2019}
{\sc Y.~M. Chan, K.~Pitchaimuthu, Q.~Wu, O.~L. Carter, G.~F. Egan, D.~R.
  Badcock, and A.~M. McKendrick}, {\em Relating excitatory and inhibitory
  neurochemicals to visual perception: A magnetic resonance study of occipital
  cortex between migraine events}, PLoS ONE, 14 (2019) (2019).

\bibitem{Cholaquidis2016}
{\sc A.~Cholaquidis, R.~Fraiman, G.~Lugosi, and B.~Pateiro-L\'{o}pez}, {\em Set
  estimation from reflected {B}rownian motion}, Journal of the Royal
  Statistical Society: Series B (Statistical Methodology), 78 (2016),
  pp.~1057--1078.

\bibitem{Choulli2016}
{\sc M.~Choulli}, {\em Applications of Elliptic Carleman Inequalities to Cauchy
  and Inverse Problems}, Springer, 2016.

\bibitem{Cirillo2019}
{\sc E.~N.~M. Cirillo, M.~Colangeli, A.~Muntean, and T.~K.~T. Thieu}, {\em A
  lattice model for active-passive pedestrian dynamics: a quest for drafting
  effects}, Mathematical Biosciences and Engineering, 17 (2019), pp.~460--477.

\bibitem{Colangeli2019}
{\sc M.~Colangeli, A.~Muntean, O.~Richardson, and T.~K.~T. Thieu}, {\em
  Modelling interactions between active and passive agents moving through
  heterogeneous environments}, vol.~1: Theory, Models and Safety Problems,, in
  G. Libelli, N. Bellomo (Eds), Crowd Dynamics, Modeling and Simulation in
  Science, Engineering and Technology, Boston, Birkhauser, Springer, 2019.

\bibitem{Prato14}
{\sc G.~Da~Prato and J.~Zabczyk}, {\em Stochastic {E}quations in {I}nfinite
  {D}imensions}, Cambridge {U}niversity {P}ress, 2014.

\bibitem{Dowling2001}
{\sc J.~E. Dowling}, {\em Neurons and Networks: An Introduction to Behavioral
  Neuroscience}, Harvard University Press, 2001.

\bibitem{Dupuis1993}
{\sc P.~Dupuis and H.~Ishii}, {\em {SDEs} with oblique reflection on nonsmooth
  domains}, Mathematical and Computer Modeling, 1 (1993), pp.~554--580.

\bibitem{Evans1998}
{\sc L.~C. Evans}, {\em Partial Differential Equations}, American Mathematical
  Society, 1998.

\bibitem{Evans2013}
{\sc L.~C. Evans}, {\em An {I}ntroduction to {S}tochastic {D}ifferential
  {E}quations}, vol.~82, American Mathematical Soc., 2013.

\bibitem{Flandoli95}
{\sc F.~Flandoli and D.~Gatarek}, {\em Martingale and stationary solutions for
  stochastic {N}avier-{S}tokes equations}, Probability Theory and Related
  Fields, 102 (1995), pp.~367--391.

\bibitem{Gilbarg1977}
{\sc D.~Gilbarg and N.~S. Trudinger}, {\em Elliptic Partial Differential
  Equations of Second Order}, vol.~224, Springer, 1977.

\bibitem{Ikeda1981}
{\sc N.~Ikeda and S.~Watanabe}, {\em Stochastic Differential Equations and
  Diffusion Processes}, Amsterdam-Tokyo: North Holland-Kodansha, 1981.

\bibitem{Jacod2004}
{\sc J.~Jacod and P.~Protter}, {\em Probability {E}ssentials}, Springer Science
  \& Business Media, 2004.

\bibitem{Karatzars2000}
{\sc I.~Karatzars and S.~E. Shreve}, {\em Brownian Motion and Stochastic
  Calculus}, Second Edition, Graduate Texts in Mathematics, Springer, 2000.

\bibitem{Kilpatrick2014}
{\sc Z.~P. Kilpatrick}, {\em {Wilson-Cowan} model}, In: Jaeger D., Jung R.
  (eds) Encyclopedia of Computational Neuroscience. Springer, New York, NY,
  (2014).

\bibitem{Lions1984}
{\sc P.~L. Lions and A.~Sznitman}, {\em Stochastic differential equations with
  reflecting boundary conditions}, Communications on Pure and Applied
  Mathematics, XXXVII (1984), pp.~511--537.

\bibitem{MarinRubio2004}
{\sc P.~Mar\'{i}n-Rubio and J.~Real}, {\em Some results on stochastic
  differential equations with reflecting boundary conditions}, Journal of
  Theoretical Probability, 17 (2004), pp.~705--716.

\bibitem{Melanson2019}
{\sc A.~Melanson and A.~Longtin}, {\em Data-driven inference for stationary
  jump-diffusion processes with application to membrane voltage fluctuations in
  pyramidal neurons}, The Journal of Mathematical Neuroscience, 9 (2019),
  pp.~1--30.

\bibitem{Przybylowicz2021}
{\sc M.~S. P.~Przyby\l{l}owicz and F.~Xu}, {\em Existence and uniqueness of
  solutions of {SDEs} with discontinuous drift and finite activity jumps},
  Statistics and Probability Letters, 174 (2021) (2021).

\bibitem{Pilipenko2012}
{\sc A.~Y. Pilipenko}, {\em On the {S}korokhod mapping for equations with
  reflection and possible jump-like exit from a boundary}, Ukrainian
  Mathematical Journal, 63 (2012), pp.~1415--1432.

\bibitem{Powanwe2021}
{\sc A.~S. Powanwe and A.~Longtin}, {\em Brain rhythm bursts are enhanced by
  multiplicative noise}, Chaos, 31(1) (2021), p.~013117.

\bibitem{Richardson2019}
{\sc O.~Richardson, A.~Jalba, and A.~Muntean}, {\em Effects of environment
  knowledge in evacuation scenarios involving fire and smoke: {A} multiscale
  modelling and simulation approach}, Fire Technology, 55 (2019), pp.~415--436.

\bibitem{Rudin1953}
{\sc W.~Rudin}, {\em Principles of Mathematical Analysis}, McGraw-Hill Book
  Company, Inc., New York-Toronto-London, 1953.

\bibitem{Sabelfeld2019}
{\sc K.~Sabelfeld}, {\em Stochastic algorithm for solving transient diffusion
  equations with a precise accounting of reflection boundary conditions on a
  substrate surface}, Appl. Math. Letters, 96 (2019), pp.~187--194.

\bibitem{Saisho1987}
{\sc Y.~Saisho}, {\em Stochastic differential equations for multi-dimensional
  domain with reflecting boundary}, Probab. Th. Rel. Fields, 74 (1987),
  pp.~455--477.

\bibitem{Situ2005}
{\sc R.~Situ}, {\em Theory of Stochastic Differential Equations with Jumps and
  Applications}, 2005.

\bibitem{Skorohod1961}
{\sc A.~V. Skorokhod}, {\em Stochastic equations for diffusion process in a
  bounded domain}, Theory of Probability and Its Applications, VI (1961),
  pp.~264--274.

\bibitem{Slominski1994}
{\sc L.~S\l{l}omi\'{n}ski}, {\em On approximation of solutions of
  multidimensional {SDE's} with reflecting boundary conditions}, Stochastic
  Processes and their Applications, 50 (1994), pp.~197--219.

\bibitem{Slominski2010}
{\sc L.~S\l{l}omi\'{n}ski and T.~Wojciechowski}, {\em Stochastic differential
  equations with jump reflection at time-dependent barriers}, Stochastic
  Processes and their Applications, 120 (2010), pp.~1701--1721.

\bibitem{Slominski1993}
{\sc L.~S{\l}ominski}, {\em On existence, uniqueness and stability of solutions
  of multidimensional {SDE's} with reflecting boundary conditions}, Annales de
  l'I.H.P. Probabilit\'es et statistiques, 29 (1993), pp.~163--198.

\bibitem{Tanaka1979}
{\sc H.~Tanaka}, {\em Stochastic differential equations with reflecting
  boundary condition in convex regions}, Hiroshima Math. J, 9 (1979),
  pp.~511--537.

\bibitem{Thieu2019}
{\sc T.~K.~T. Thieu, M.~Colangeli, and A.~Muntean}, {\em Weak solvability of a
  fluid-like driven system for active-passive pedestrian dynamics}, Nonlinear
  Studies, 26 (2019), pp.~991--1006.

\bibitem{Wells2006}
{\sc C.~G. Wells}, {\em A stochastic approximation scheme and convergence
  theorem for particle interactions with perfectly reflecting boundary
  conditions}, Monte Carlo Methods and Applications, 12 (2006), pp.~291--342.

\bibitem{Wilson1972}
{\sc H.~R. Wilson and J.~D. Cowan}, {\em Excitatory and inhibitory interactions
  in localized populations of model neurons}, Biophysical journal, 12 (1972),
  pp.~1--24.

\bibitem{Yamada1971}
{\sc T.~Yamada and S.~Watanabe}, {\em On the uniqueness of solutions of
  stochastic differential equations}, J. Math. Kyoto Univ., 11 (1971),
  pp.~1415--1432.

\end{thebibliography}
\end{document}